\documentclass[11pt, reqno]{amsproc} 

\usepackage{geometry}
\usepackage{color}
\usepackage{verbatim}
\usepackage[utf8x]{inputenc}
\usepackage{t1enc}
\usepackage[english]{babel}

\usepackage{amsmath,amssymb,amsthm}
\usepackage{mathrsfs,esint}
\usepackage{graphicx,wrapfig}
\usepackage[format=hang]{caption}

\newcommand*{\R}{{\mathbb R}}

\providecommand*{\vint}[1]{\mathchoice
          {\mathop{\vrule width 5pt height 3 pt depth -2.5pt
                  \kern -9pt \kern 1pt\intop}\nolimits_{\kern -5pt{#1}}}
          {\mathop{\vrule width 5pt height 3 pt depth -2.6pt
                  \kern -6pt \intop}\nolimits_{\kern -3pt{#1}}}
          {\mathop{\vrule width 5pt height 3 pt depth -2.6pt
                  \kern -6pt \intop}\nolimits_{\kern -3pt{#1}}}
          {\mathop{\vrule width 5pt height 3 pt depth -2.6pt
                  \kern -6pt \intop}\nolimits_{\kern -3pt{#1}}}}

\DeclareMathOperator{\diam}{diam}

\numberwithin{equation}{section}
\theoremstyle{plain}
\newtheorem{thm}[equation]{Theorem}

\newtheorem{lem}[equation]{Lemma}

\theoremstyle{definition}

\newtheorem{defn}[equation]{Definition}

\newtheorem{remark}[equation]{Remark}

\begin{document}

\title[Carrasco's theorem on quasisymmetric maps]
{On Carrasco Piaggio's theorem characterizing quasisymmetric maps from compact doubling spaces to Ahlfors regular
spaces} 
\author{Nageswari Shanmugalingam}
\address{Department of Mathematical Sciences, P.O.~Box 210025, University of Cincinnati, Cincinnati, OH~45221-0025, U.S.A.}
\email{shanmun@uc.edu}
\thanks{This material is motivated by the series of learning seminars 
during the author's stay at the 
Mathematical Sciences Research Institute (MSRI, Berkeley, CA) while she was resident there as a member of  the program
\emph{Analysis and Geometry in Random Spaces} which is 
 supported by the National Science Foundation (NSF U.S.A.) under Grant No. 1440140, during Spring 2022. 
 The author thanks MSRI for its kind hospitality, and Mario Bonk, Mathav Murugan, 
 and Pekka Pankka for valuable discussions on~\cite{Car} and for comments that helped improve the exposition of the paper.
 The author's work is partially supported by the NSF (U.S.A.) grant DMS~\#2054960.}
\maketitle

\begin{center}
Dedicated to the memory of Professor David R.~Adams.\\
\end{center}

\begin{abstract}
In this note we deconstruct and explore the components of a theorem of Carrasco Piaggio, which relates Ahlfors regular
conformal gauge of a compact doubling metric space to weights on Gromov-hyperbolic fillings of the metric space.
We consider a construction of hyperbolic filling that is simpler than the one considered by Carrasco Piaggio, and
we determine the effect of each of the four properties postulated by Carrasco Piaggio on the induced metric on
the compact metric space. 
\end{abstract}

\noindent
    {\small \emph{Key words and phrases}:
Gromov hyperbolic filling, uniformization, 
metric space, quasisymmetry, Ahlfors regular, uniformly perfect, conformal change in metric.}

\medskip

\noindent
    {\small Mathematics Subject Classification (2020):
Primary:
30L05.
Secondary:
30L10, 51F30, 53C23.
}

\section{Introduction}

Within the class of metric spaces, those that are Gromov hyperbolic possess the properties of negative curvature
at large scale but are not concerned with small-scale behavior; and as such, Gromov hyperbolicity is stable under
biLipschitz changes in the metric (unlike Alexandrov curvature conditions). First proposed as a structure useful in the study of
Cayley graphs of hyperbolic groups~\cite{Gro}, the study of Gromov hyperbolic spaces was subsequently found
to be useful in the study of potential theory~\cite{BBS2}. It is also connected to the study of 
metric geometry, as there is a close connection between
Gromov hyperbolic spaces and uniform domains~\cite{BHK}, and between rough quasiisometries between Gromov
hyperbolic spaces and quasisymmetries between their visual boundaries. It is this latter connection that
is explored further in~\cite{Car}, and is based on the fact that every compact doubling metric space is the
boundary of a Gromov hyperbolic space, called hyperbolic filling, of the space.
Now there is extensive literature on various uses of hyperbolic filling, dating back to the seminal paper of
Gromov~\cite[page~95]{Gro}, and made explicit in~\cite{BBS, BoSa, BoSc, BP, BuSch, Car, Ibr, KM, KSS, Kig, Kig2, L};
these are merely a sampling of current literature on the topic of Gromov hyperbolicity and hyperbolic filling.

During the author's stay at MSRI, there was an extensive discussion of the paper~\cite{Car} characterizing metrics
on a compact space that are quasisymmetrically equivalent and at least one of them an Ahlfors regular metric. 
The results of~\cite{Car} were of great interest to many participants at MSRI. However,
the complicated system of parameters used there made it difficult to see the underlying beautiful ideas in~\cite{Car}.
The goal of the current note
is to deconstruct the role of some of the parameters in used there, and to eliminate others, thus providing a
simplified expository 
discourse on parts of~\cite{Car}. The focus is on~\cite[Theorem~1.1]{Car}. The following theorem is the result of 
exploring the role of each of the conditions~(H1)--(H4) assumed in~\cite{Car}.

\begin{thm}\label{thm:main}
Let $(Z,d_Z)$ be a compact doubling metric space. Fixing $\alpha\ge 2$, and $\tau\ge 2\alpha^2+1$, we choose
a hyperbolic filling $X$ of $Z$ associated with the parameters $\alpha$ and $\tau$ as in Definition~\ref{rem:propties-hyp}.
\begin{enumerate}
\item[{\bf I.}] Suppose that $\rho:X\to(0,1)$, and consider the function $d_\rho$ on $X\times X$ associated with $\rho$ as in
Definition~\ref{def:rho-metr}. 
\begin{enumerate}
\item If $\rho$ satisfies Condition~(H1) of Definition~\ref{def:H1-3}, then $d_\rho$ is a metric on $X$, with
$(X,d_\rho)$ a locally compact, non-complete metric space. Let $\partial_\rho X:=\overline{X}\setminus X$, with
$\overline{X}$ the completion of $X$ with respect to the metric $d_\rho$.
\item If $\rho$ satisfies Conditions~(H1) and~(H3) of Definition~\ref{def:H1-3}, then there is a homeomorphism
$\Phi:Z\to\partial_\rho X$ and positive constants $c, C$ such that for every $x,y\in Z$ we have
\[
 c\, d_Z(x,y)^{\tau_-}\le d_\rho(\Phi(x),\Phi(y))\le C\, d_Z(x,y)^{\tau_+}
\]
with 
\[
\tau_-:=\frac{\log(\eta_-)}{\log(1/\alpha)}, \qquad \tau_+:=\frac{\log(\eta_+)}{\log(1/\alpha)}.
\]
\item If $\rho$ satisfies Conditions~(H1), (H2), and~(H3) of Definition~\ref{def:H1-3}, then the map $\Phi$ is a
quasisymmetry.
\item If $\rho$ satisfies Conditions~(H1), (H2), and~(H3) of Definition~\ref{def:H1-3} and Condition~(H4) of
Definition~\ref{def:H4}, then $(\partial_\rho X,d_\rho)$ is Ahlfors $p$-regular.
\end{enumerate}
\item[{\bf II.}] Conversely, suppose that $Z$ is $C_U$-uniformly perfect for some $C_U>2$, and $\alpha>C_U^3$ with
$\tau\ge \max\{\alpha^2+1,2C_U^3(C_U^2-4)^{-1}\}$. If
$\theta$ is any metric on $Z$ for which $(Z,\theta)$ is Ahlfors $p$-regular and 
is quasisymmetric to $(Z,d_Z)$, then there exists a function $\rho:X\to(0,1)$ that 
satisfies Conditions~(H1), (H2), (H3), and~(H4).
\end{enumerate}
\end{thm}

\begin{remark}
Note in the above theorem that in Part~{\bf I.}~we \emph{do not} require $(Z,d_Z)$ to be uniformly perfect; then,
Conditions~(H1)---(H3) do not imply uniform perfectness either 
(and indeed, the choice of $\rho$ as the constant function $\rho(x)=1/\alpha$ satisfies Conditions~(H1)---(H3)
with the resulting quasisymmetry a biLipschitz map, see~\cite{BBS}, claim~{\bf I.}(b) of Theorem~\ref{thm:main} above,
or Theorem~\ref{thm:biHolderBdy} below); however, Conditions~(H1)---(H4) \emph{together}
imply that $(Z,d_Z)$ must be uniformly perfect. Thus {\bf I.}(a)---(c) on their own are not explicitly covered in~\cite{Car}, for
Carrasco Piaggio~\cite{Car} does explicitly require $Z$ to be uniformly perfect (see~\cite[Section~2.1]{Car}), 
that assumption is also implicit in the four conditions
together (see Lemma~\ref{lem:unifPerf} below), and conversely, if $(Z,d)$ is quasisymmetric to $(Z,\theta)$ with $\theta$
Ahlfors $d$-regular, then necessarily $(Z,d)$ is uniformly perfect as well. Interestingly also, in~\cite[page~507, (2.8)]{Car},
Carrasco Piaggio requires $\tau\ge 32$ (there $\tau$ is denoted $\lambda$) and then require 
$\alpha\ge 6\kappa^2\, \max\{\tau, C_U\}$ (with $\alpha$ denoted as $a$ and $C_U$ denoted as $K_P$ in~\cite{Car}).
The parameter $\kappa$ is an additional one associated with the construction of hyperbolic filling as given in~\cite{Car}; with
the simplified construction as considered in this note and in~\cite{BBS}, we have $\kappa=1$. Thus, in~\cite{Car} the
parameter $\alpha$ depends on the choice of $\lambda$ and $C_U$, but in our note $\tau$ depends on the choice of
$\alpha$ while in Part~{\bf II.}, both $\alpha$ and $\tau$ depend on $C_U$ as well. 

As pointed out above, when considering only the conditions~(H1)---(H3), the metric space $(Z,d)$ need not be uniformly
perfect, 
but still the quasisymmetry $\Phi$ obtained in Section~\ref{Sec:5} is necessarily a power quasisymmetry. Since there are 
compact doubling metric spaces and quasisymmetries on them that are \emph{not} power quasisymmetries
(see for example the discussion in~\cite{Hei}), it follows that not all quasisymmetries on a doubling space are 
obtained using the method of Carrasco Piaggio~\cite{Car}. 
\end{remark}
 
Section~\ref{Sec:2} is devoted to describing the construction of hyperbolic
filling, and the last five sections of this note are devoted
to the proof of the claims of the theorem.  
We choose to use the construction of hyperbolic filling from~\cite{BBS} for its simplicity in relation to the
one used in~\cite{Car}. While the construction in~\cite{Car} (see also~\cite{Kig2}) 
gives greater flexibility to the choice of sets and vertices,
it is perhaps this very flexibility that makes it difficult to see what the effect of the conditions~(H1)--(H4) are, and so
we chose the simpler version given in~\cite{BBS}. However, the ideas and basic premises are as in~\cite{Car}.

In Section~\ref{Sec:3} the conditions~(H1)---(H3) are discussed and {\bf I.}(a) of Theorem~\ref{thm:main} is proved,
while in Section~\ref{Sec:4} the claim {\bf I.}(b) of the theorem is verified. Section~\ref{Sec:5} is devoted to the 
proof of {\bf I.}(c) of Theorem~\ref{thm:main}, and the discussion in Section~\ref{Sec:6} completes the proof of
the part {\bf I.} of Theorem~\ref{thm:main}. The focus of Section~\ref{Sec:7} is to prove part {\bf II.} of
Theorem~\ref{thm:main}. In Section~\ref{sec:alternate-cond} we list a set of four conditions that parallel the conditions
of Carrasco Piaggio~\cite{Car}, but couched from the perspective of densities on a metric space that lead to
conformal changes in the metric. We end that section by posing a query regarding an Adams-type inequality~\cite{Ad, AH, Mak},
which is known to hold in the case that the function $\rho$ is the constant function $\rho(x)=1/\alpha$.

\vskip .3cm

\noindent {\bf Acknowledgement:} This material is motivated by the series of learning seminars 
during the author's stay at the 
Mathematical Sciences Research Institute (MSRI, Berkeley, CA) while she was resident there as a member of  the program
\emph{Analysis and Geometry in Random Spaces} which is 
 supported by the National Science Foundation (NSF U.S.A.) under Grant No. 1440140, during Spring 2022. 
 The author thanks MSRI for its kind hospitality, and Mario Bonk, Mathav Murugan, 
 and Pekka Pankka for valuable discussions on~\cite{Car} and for comments that helped improve the exposition of the paper.
 The author's work is partially supported by the NSF (U.S.A.) grant DMS~\#2054960.

\section{Construction of hyperbolic filling}\label{Sec:2}

Recall that a metric space $(Z,d_Z)$ is \emph{metric doubling} if there is a positive integer $N$ such that for each
$z\in Z$ and $r>0$, if $A\subset B(z,r)$ such that $d_Z(x,y)\ge r/2$ whenever $x,y\in A$ with $x\ne y$, then there are
at most $N$ number of elements in $A$.

In this note, $(Z,d_Z)$ is a compact metric space, such that it is a metric doubling space.  
Later we will also assume
that $Z$ is uniformly perfect, that is, there is some $C_U>1$ such that for each $z\in Z$ and $0<r<\diam(Z)/2$,
the annulus $B_{d_Z}(z,r)\setminus B_{d_Z}(z,r/C_U)$ is non-empty; however, for now we do not need this assumption.
We will, however, also assume that $0<\diam(Z)<1$ without loss of generality (as we are not interested in singleton
metric spaces).

Constructions of hyperbolic fillings of compact doubling metric spaces can be found for example 
in~\cite{BP, BoSa, BoSc, BuSch, Car, BBS}. The version we give here is that of~\cite{BBS}. The obtained graph in this
construction, when equipped with the path metric, is Gromov hyperbolic; however, this fact is not essential for the discussion
in this note, as we turn the graph into a metric graph by adding unit interval edges to connect neighboring pairs of vertices
and then use path integrals to directly obtain a metric on the graph; hence its boundary can be realized via a metric
completion rather than as the visual boundary of a Gromov hyperbolic space. For this reason, we do not devote
space to discussing Gromov hyperbolicity here. We refer the interested reader to the discussion in~\cite[Section~3]{BBS}.

\begin{defn}\label{rem:propties-hyp}
By a rescaling of the metric if necessary, we may assume without loss of generality that $0<\diam(Z)<1$.
We fix $\alpha\ge 2$ and $\tau>1$, and for each non-negative integer $n$ we set $A_n$ to be a maximal 
$\alpha^{-n}$-separated subset of $Z$, that is, if $z,w\in Z$ with $z\ne w$, then $d_Z(z,w)\ge \alpha^{-n}$, and
$Z=\bigcup_{w\in A_n}B_{d_Z}(w,\alpha^{-n})$. We can, via an inductive construction, ensure that $A_n\subset A_{n+1}$ for
each non-negative integer $n$. We set $V=\bigcup_{n=0}^\infty A_n\times\{n\}$. The set $V$ is the vertex set of the metric
graph $X$ to be constructed next. We do this construction as follows. The vertex $w_0=(x_0,0)$, with $x_0\in A_0$,
will play the role of a root of the graph.
\begin{enumerate}
\item Two vertices $v_1=(z_1,n_1), v_2=(z_2,n_2)\in V$ are neighbors, denoted $v_1\sim v_2$, if 
$v_1\ne v_2$ and either $n_1=n_2$ with $B_{d_Z}(z_1,\tau\alpha^{-n_1})\cap B_{d_Z}(z_2,\tau\alpha^{-n_2})\ne \emptyset$, or else
$n_1=n_2\pm 1$ and $B_{d_Z}(z_1,\alpha^{-n_1})\cap B_{d_Z}(z_2,\alpha^{-n_2})\ne \emptyset$. 
\item We turn $V$ into a metric 
graph $X$ by gluing a unit-length interval to each pair of neighboring vertices. 
\item We call a vertex $v_2=(z_2,n_2)$ a
child of a vertex $v_1=(z_1,n_1)$ if $v_1\sim v_2$ and $n_2=n_1+1$; we also then say that the edge
$[v_1,v_2]$ is a vertical edge. If $[v_1,v_2]$ is a vertical edge, then necessarily $d_Z(z_1,z_2)<\alpha^{-n_1}+\alpha^{-n_2}$,
and so with $n=\min\{n_1,n_2\}$, we have that $d_Z(z_1,z_2)<\alpha^{1-n}$ (we use our choice of  $\alpha\ge 2$ here).
\item If $v_1\sim v_2$ with $n_1=n_2$, then we say
that the edge $[v_1,v_2]$ is a horizontal edge. In this case we have that $d_Z(z_1,z_2)<\tau \alpha^{1-n_1}$.
\item We say that a point $x\in X$ is a descendant of a point $y\in X$ if 
there is a vertically descending path \emph{from} $y$ to $x$. 
\item A vertex $v$ is said to be a \emph{common ancestor} of
two points $x,y\in X$ if there are two vertically descending paths, one from $v$ to $x$ and the other from $v$ to $y$.
\item Also, given a vertex $v=(z,n)\in V$, we set 
\[
\Pi_1(v)=z\ \text{ and }\ \Pi_2(v)=n.
\]
\item If $\tau\ge 1+1/\alpha$ and $(z,n), (x_1,n-1), (x_2,n-1)\in V$ such that $(z,n)\sim (x_i,n-1)$ for $i=1,2$, then 
$(x_1,n-1)\sim(x_2,n-1)$.
\item Thanks to the doubling property, there is a constant $C\ge 1$, depending only on the doubling constant related to
the metric doubling property of $(Z,d_Z)$ and the choice of $\alpha, \tau$, such that for each positive integer $n$ we have
$\sum_{x\in A_n}\chi_{B_{d_Z}(x,\tau\alpha^{-n})}\le C$ pointwise everywhere on $Z$.
\item Suppose that $\cdots\sim (x_{n+1},n+1)\sim (x_{n},n)\sim (y_n,n)\sim (y_{n+1},n+1)\sim\cdots$ is a path in the graph,
allowing for the possibility that $x_n=y_n$ by a slight abuse of notation above, we see that for each $k\ge n$,
$d_Z(x_k,x_{k+1})\le \alpha^{-k}+\alpha^{-k-1}\le \alpha^{1-k}$ (we use the choice $\alpha\ge 2$ here). With similar estimates
holding for $d(y_k,y_{k+1})$, we see that the two sequences $(x_k)_{k\ge n}$ and $(y_k)_{k\ge n}$ are Cauchy sequences in $Z$,
converging to points denoted $x$ and $y$ respectively. We see that then for each $j\ge n$, 
\[
d_Z(x,x_j)\le \sum_{n=j}^\infty \alpha^{1-n}=\frac{\alpha^{2-j}}{\alpha-1},
\]
with a similar estimate holding for $d_Z(y,y_j)$. Suppose that $x\ne y$. With $n_{xy}$ a non-negative integer such that
$\alpha^{-n_{xy}}<d_Z(x,y)\le \alpha^{1-n_{xy}}$, and $j_0$ a non-negative integer such that $\alpha^{-j_0}<\tau-1\le \alpha^{1-j_0}$,
we have that
\[
\alpha^{-n_{xy}}<d_Z(x,y)\le d_Z(x,x_n)+d_Z(x_n,y_n)+d_Z(y_n,y)\le \frac{2\alpha^{2-j}}{\alpha-1}+2\tau\alpha^{-n}\le \alpha^{3+j_0-n}.
\]
It follows that 
\begin{equation}\label{eq:go-in}
n\le 3+j_0+n_{xy}.
\end{equation}
\item Given a vertex $v=(x,n)\in V$, there is a vertically descending geodesic ray 
$w_0=v_0\sim v_1\sim\cdots\sim v_k\sim\cdots$ with $v_k=v$ for each $k\ge n$. This is done by choosing $v_k=(x_k,k)$ for
$k=1,\cdots, n-1$ such that $x_k\in A_k$ with $d_Z(x,x_k)\le \alpha^{-k}$.
\end{enumerate}
\end{defn}

Note that $A_0$ has only one point by our hypothesis that $\diam(Z)<1$. The vertex $w_0=(x_0,0)$ plays 
a distinguished role in the graph corresponding to $x_0\in A_0$. If $z\in A_{n+1}\setminus A_n$, then by the maximality
of $A_n$ there is a point $w_z\in A_n$ such that $d_Z(z,w_z)<\alpha^{-n}$, and so $(z,n+1)\sim(w_z,n)$; therefore it is
easy to see that $X$ is path-connected. While this construction is not exactly the one considered in~\cite{Car}, it is
in the spirit of~\cite{Car} and is the one used in~\cite{BBS}. From~\cite[Theorem~3.4]{BBS} we know that $X$ is
Gromov hyperbolic, with hyperbolicity constant depending solely on $\alpha$ and $\tau$.

Larger the choice of $\tau$ is, the greater the number of horizontal edges. Since $Z$ is doubling, each vertex $v\in V$
has a uniformly bounded degree, with the upper bound on the degree depending solely on the doubling constant associated
with $\nu$ and the parameters $\alpha$ and $\tau$. 
\emph{Henceforth, we will fix $\alpha\ge 2$ and $\tau\ge 1+\tfrac{1}{\alpha}$.} The condition on $\tau$ ensures that
the conclusion of~(8) above holds.

\section{Weighted uniformization metric and three conditions}\label{Sec:3}

Since $\diam(Z)>0$, the graph $X$, equipped with the path metric $d_X$, is necessarily unbounded.
In this section we consider a family of uniformizations, each dampening the metric $d_X$ at locations far from the
root vertex $w_0$, so that the dampened metric on $X$ turns $X$ into a bounded non-complete metric space. The
principal object of study in this note is the boundary of the damped space, as it is in~\cite{Car}.

\begin{defn}\label{def:H1-3}
We consider a function $\rho:V\to\R$ that satisfies the following
conditions (using the labels from~\cite{Car}):
\begin{enumerate}
\item[{(H1)}] There exist $0<\eta_-\le \eta_+<1$ such that $\rho:X\to[\eta_-,\eta_+]$.
\item[{(H2)}] There is a constant $K_0>0$ so that
if $v_1,v_2\in V$ with $v_1\sim v_2$, and
if $w_0\sim w_1\sim\cdots\sim w_k=v_1$ and 
$w_0=u_0\sim u_1\sim\cdots\sim u_n=v_2$ are vertical edges, then
\[
\pi(v_1):=\prod_{j=0}^k\rho(w_j)\le K_0\prod_{j=0}^n\rho(u_j)=:K_0\, \pi(v_2).
\]
This also defines $\pi:V\to(0,\infty)$. We extend $\pi$ to all of $X$ by setting $\pi(x)=t\pi(v_1)+(1-t)\pi(v_2)$
when $x$ is a non-vertex point in the edge $[v_1,v_2]$, and $t$ denotes the distance from $x$ to the vertex $v_1$.
\item[{(H3)}] 
There is a constant $K_1>0$ satisfying the following condition. Whenever $x,y\in X$ with $x, y$ belonging to different
edges of $X$, there are two vertically descending paths $w_0=v_0\sim v_1\sim\cdots\sim v_k$,
$w_0=u_0\sim u_1\sim\cdots\sim u_n$ with $x\in [v_{k-1},v_k]$, $y\in [u_{k-1},u_k]$. Let $v_{xy}$ denote the vertex
in the path $w_0=v_0\sim v_1\sim\cdots\sim v_k$ with largest possible value of $\Pi_2(v_{xy})$ such that 
either $v_{xy}=u_{\Pi_2(v_{xy})}$ or else $v_{xy}\sim u_{\Pi_2(v_{xy})}$. For every path 
$\gamma$ in $X$ with end points $x$ and $y$, we must have
\[
\int_\gamma \pi(\gamma(t))\, dt\ge K_1^{-1} \pi(v_{xy}).
\]
\end{enumerate}
\end{defn}

\begin{remark}\label{rem:pi-is-welpsd}
Note that in Condition~(H2), if we have $v_2=v_1$ instead of $v_2\sim v_1=(x_v,n)$, then $k=n$ and necessarily
\begin{align*}
d_Z(\Pi_1(w_{n-1}),\Pi_2(u_{n-1}))&\le d_Z(\Pi_1(w_{n-1}),x_v)+d_Z(x_v,\Pi_2(u_{n-1}))\\
& \le 2\left[\alpha^{-(n-1)}+\alpha^{-n}\right]\le 4\alpha^{1-n}.
\end{align*}
It follows that if $\alpha\ge 2$ and $\tau\ge 2\alpha^2+1>4$, then $w_{n-1}\sim u_{n-1}$. Hence from~(H2) we have 
that $\pi(v)$, up to the ambiguity of the multiplicative constant $K_0$, is well-defined in that the choice of the descending
path used to define $\pi(v)$ is not crucial.
\end{remark}

\begin{remark}\label{rem:com-elders}
If $x\in X$, we can find paths $=w_0=v_0\sim v_1\sim\cdots$ in $X$ so that for each positive integer $n$ we have
that $\Pi_1(v_n)\in A_n$ with $d_Z(x,\Pi_1(v_n))<\alpha^{-n}$. Let $w_0\sim w_1\sim\cdots$ be another such path
associated with a point $y\in X$, and let $v_{xy}$ be the vertex point in the path $\{v_n\, :\, n=0,1,\cdots\}$ that is
a neighbor of $w_{\Pi_2(v_{xy})}$ 
such that $\Pi_2(v_{xy})$ be the largest possible (i.e., the latest common ancestor). Then from~(H3) above, 
when $\gamma$ is the concatenation of the curves from $v_{xy}$ to $x$ and to $y$ respectively via the sequences
$(v_n)_{n\ge \Pi_2(v_{xy})}$, $v_{xy}\sim w_{\Pi_2(v_{xy})}$, and $(w_n)_{n\ge \Pi_2(v_{xy})}$, we have that
\[
\int_\gamma \pi(\gamma(t))\, dt\ge K_1^{-1}\, \pi(v_{xy}),
\]
see point~(10) of Definition~\ref{rem:propties-hyp} above.
\end{remark}

We will use the weight $\pi$ 
as the conformal density, to 
modify the metric on the graph $X$ from the path metric to the metric $d_\rho$.

In~\cite{Car} a fourth condition is also required, but we will not consider that condition until the penultimate section of this note. We
postpone its definition to that section, see Definition~\ref{def:H4} below.

\begin{defn}\label{def:rho-metr}
Let $d_\rho:X\times X\to[0,\infty)$ be given as follows. For $x,y\in X$, we set
\[
d_\rho(x,y)=\inf_\gamma\int_\gamma\pi(\gamma(t))\, dt,
\]
where the infimum is over all paths $\gamma$ in $X$ with end points $x$ and $y$. We only consider paths that are arc-length 
parametrized with respect to the graph metric $d_X$.
\end{defn}

\begin{lem}\label{lem:metric}
Suppose that $\rho$ satisfies Condition~(H1). Then
$d_\rho$ is a metric on $X$. Moreover, $(X,d_\rho)$ is locally compact, non-complete metric space.
\end{lem}

\begin{proof}
Let $x,y\in X$ with $x\ne y$ and $\gamma$ be a curve in $X$ with end points $x$ and $y$.
If $x$ and $y$ belong to the same edge $[v_1,v_2]$ in $X$, then any curve $\gamma$ connecting $x$ to $y$ has to 
contain a subcurve of $d_X$-length at least $d_X(x,y)$ that lies in the subgraph obtained by adding the edges that
have either $v_1$ or $v_2$ as a vertex-endpoint. Hence, with $n=\max\{\Pi_2(v_1),\Pi_2(v_2)\}$, we have that
\[
\int_\gamma\pi(\gamma(t))\, dt\ge \eta_-^{n+1}\, d_X(x,y)>0,
\]
and taking the infimum over all curves $\gamma$ gives $d_\rho(x,y)\ge \eta_-^{n+1}d_X(x,y)>0$.

Next, suppose that $x$ and $y$ belong to different edges.
Then any curve $\gamma$ connecting $x$ to $y$ has to have a sub-curve of positive $d_X$-length that passes through
a vertex $v\ne x$ such that $v$ is a neighbor of one of the two vertices that make up the edge $x$ lies in. It follows that
$\Pi_2(v)\le n_x+1$, with $n_x$ a positive integer that depends solely on $x$. Hence
\[
\int_\gamma\pi(\gamma(t))\, dt\ge \pi(v)d_X(x,v)\ge \eta_-^{n_x+1}\, d_X(x,v)>0.
\]
Taking the infimum over all $\gamma$ gives $d_\rho(x,y)\ge \eta_-^{n_x+1}\, d_X(x,v)>0$. 
Thus, in both cases we have that if $x\ne y$ then $d_\rho(x,y)>0$. The triangle inequality and symmetry follow immediately
from the definition of $d_\rho$, and so $d_\rho$ is a metric on $X$.

From the first paragraph of this proof, we know that for each vertex $v$, the subgraph made up of all the edges that have
$v$ as an end-point is a compact subset of $(X,d_\rho)$, and moreover, $v$ is in the $d_\rho$-interior of this subgraph.
Hence $(X,d_\rho)$ is locally compact. 

Finally, for each non-negative integer $n$ we set $w_n=(x_0,n)$. Then $w_0\sim w_1\sim\cdots \sim w_n\sim w_{n+1}\sim\cdots$,
and as the edge $[w_n,w_{n+1}]$ is a path connecting the two vertices $w_n$ and $w_{n+1}$, we see that
\begin{equation}\label{eq:easyUbound}
d_\rho(w_n,w_{n+1})\le \eta_+^n.
\end{equation}
As $0<\eta_+<1$, it follows that $(w_n)_n$ is a Cauchy sequence in $(X,d_\rho)$. This sequence does not converge to any element in 
$X$. Therefore $(X,d_\rho)$ is non-complete.
\end{proof}

If $\rho$ is the constant function $\rho(x)=1/\alpha$, where $\alpha$ (together with $\tau$) is the parameter used in constructing
the hyperbolic filling $X$ of $Z$, then $\pi(v)\approx \alpha^{-n}$ where $n=\Pi_2(v)$. Therefore, from~\cite[Proposition~4.4]{BBS}
we know that $\partial_\rho X=:\overline{X}\setminus X$ is biLipschitz equivalent to $Z$. Here, the completion $\overline{X}$ 
is taken with respect to the metric $d_\rho$. \emph{We will not need this information for our discussion in this note, and so we do 
not elaborate on this further} but refer the interested reader to~\cite{BBS}.

In Lemma~\ref{lem:metric} only~(H1) played a role. In the next section Conditions~(H1) and~(H3) together will play a key role, but
Condition~(H2) will not.

\section{Bi-H\"older property}\label{Sec:4}

Recall from the pervious section that $(X,d_\rho)$ is locally compact but not complete. We set 
$\partial_\rho X:=\overline{X}\setminus X$, where $\overline{X}$ is the completion of $X$ with respect
to $d_\rho$. As $X$ is locally compact with respect to $d_\rho$ (see Lemma~\ref{lem:metric}), 
it follows that $X$ is an open subset of $\overline{X}$.

\vskip .3cm

As shown in~\cite[Proposition~4.1]{BBS}, if $\eta_- <1/\alpha$, then there is no guarantee that 
$\partial_\rho X$ is even homeomorphic to $Z$; hence if $\eta_-<1/\alpha$, then Condition~(H3) becomes vital in 
obtaining that $\partial_\rho X$ is homeomorphic to $Z$.

We now construct a natural map $\Phi:Z\to\partial_\rho X$ as follows.

\begin{defn}\label{def:phi}
For $z\in Z$ and for each positive integer $n$ we can find $v_n\in V$ such that
with $x_n=\Pi_1(v_n)\in A_n$ and $\Pi_2(v_n)=n$, with $d_Z(x_n,z)<\alpha^{-n}$. Note that then 
$z\in B_{d_Z}(x_n,\alpha^{-n})\cap B_{d_Z}(x_{n+1},\alpha^{-(n+1)})$, and so $v_n\sim v_{n+1}$, and hence
$w_0=v_0\sim v_1\sim\cdots\sim v_n\sim v_{n+1}\sim\cdots$ is a vertically descending path in $X$, with
$\pi(v_n)\le \eta_+^n$. Hence the sequence $(v_n)_n$ is a Cauchy sequence in $(X,d_\rho)$, for
we have that $d_\rho(v_n,v_{n+1})\le 2\eta_+^n$, see~\eqref{eq:easyUbound}.
We set $\Phi(x)$ to be the class of all Cauchy sequences in $(X,d_\rho)$ that are equivalent to this Cauchy sequence. 

To see that $\Phi$ is well-defined, suppose that $y_n\in A_n$ for each positive integer $n$ such that 
$d_Z(x,y_n)<\alpha^{-n}$. Then $(y_n,n)\sim v_n$, because $z\in B_{d_Z}(x_n,\alpha^{-n})\cap B_{d_Z}(y_n,\alpha^{-n})$.
As above, the sequence $((y_n,n))_n$ is also Cauchy with respect to the metric $d_\rho$, but also 
$d_\rho(v_n,(y_n,n))\le \eta_+^n+\eta_+^{n+1}$, and so the two Cauchy sequences are equivalent with respect to
the metric $d_\rho$. Thus, $\Phi:Z\to\partial_\rho X$ is well-defined. 
\end{defn}

\begin{thm}\label{thm:biHolderBdy}
Suppose that $\rho$ satisfies Conditions~(H1) and~(H3). 
Then $\Phi$ is a homeomorphism with 
\begin{equation}\label{eq:biHolder}
C^{-1}\, d_Z(x,y)^{\tau_-}\le d_\rho(\Phi(x),\Phi(y))\le C\, d_Z(x,y)^{\tau_+}
\end{equation}
for each $x,y\in Z$, where
\[
\tau_-:=\frac{\log(\eta_-)}{\log(1/\alpha)}, \qquad \tau_+:=\frac{\log(\eta_+)}{\log(1/\alpha)}.
\]
Moreover, $d_\rho(\Phi(x),\Phi(y))\approx \pi(v_{xy})$.
\end{thm}

We remind the reader that the root of $X$ is denoted $w_0=(x_0,0)$.

\begin{proof}
Let $j_0$ be the unique integer such that $\alpha^{-j_0}<\tau-1\le \alpha^{1-j_0}$.

We first aim to prove~\eqref{eq:biHolder}.
Let $x,y\in Z$, and choose a positive integer $n_{xy}$ such that
$\alpha^{-n_{xy}}<d_Z(x,y)\le \alpha^{1-n_{xy}}$. We fix a path $w_0=v_0\sim v_1\sim\cdots$ such that for each non-negative integer
$n$ we have that $\Pi_1(v_n)\in A_n$ with $d_Z(\Pi_1(v_n),x)\le \alpha^{-n}$. Let $v_0=w_0\sim w_1\sim\cdots$ be a corresponding
choice of descending sequence with respect to $y$. We claim that for each non-negative integer $n$ with $n\le n_{xy}-j_0-1$,
either $v_n=w_n$ or $v_n\sim w_n$. To this end, we assume that $v_n\ne w_n$ and $1\le n\le n_{xy}-j_0-1$,  
for otherwise there is nothing to prove. 
Since $d_Z(x,\Pi_1(v_n))\le \alpha^{-n}$ and $d_Z(y,\Pi_1(w_n))\le \alpha^{-n}$, and
as $n\le n_{xy}-j_0-1$, it follows that 
\[
d_Z(x,\Pi_1(w_n))\le \alpha^{-n}+\alpha^{1-n_{xy}}\le \alpha^{-n}(1+\alpha^{-j_0})<\tau\alpha^{-n}.
\]
It follows that $x\in B_{d_Z}(\Pi_1(v_n),\alpha^{-n})\cap B_{d_Z}(\Pi_1(w_n),\tau\alpha^{-n})$, and so $v_n\sim w_n$. 
Next we claim that
if $n$ is a positive integer with $v_n\sim w_n$, then $n\le n_{xy}+(-j_0)_++3$. Indeed, we have that 
\[
\alpha^{-n_{xy}}<d(x,y)\le d(x,\Pi_1(v_n))+d(y,\Pi_1(w_n))+2\tau\alpha^{-n}\le 2(1+\tau)\alpha^{-n}\le \alpha^{1-n}(1+\tau), 
\]
with $1+\tau\le 2\le 1+\alpha^{1-j_0}\le \alpha^{3-n}$ if $j_0\ge 0$, and $1+\tau\le \alpha^{3-n-j_0}$ if $j_0<0$.
From this we obtain $n+(-j_0)_+-3<n_{xy}$. As $n$ and $n_{xy}$ are integers, it follows that $n\le n_{xy}+(-j_0)_++3$.

We now fix a choice of sequences $v_n, w_n$, $n=0,1,\cdots$ as above corresponding to the points $x,y\in Z$, and 
let $F[x,y]$ denote the collection of all vertices $v_n$ for which $v_n\sim w_n$ or $v_n=w_n$. 
Let $v_{xy}$ be the vertex in $F$ for which
$\Pi_2(v_{xy})=\max\{\Pi_2(v)\, :\, v\in F[x,y]\}$. For symmetry's sake, we also set $w_{xy}$ to be from the sequence
corresponding to $y$ such that $w_{xy}=w_{\Pi_2(v_{xy})}$. 

Recall that $j_0$ is the integer such that $\alpha^{-j_0}<\tau-1\le \alpha^{1-j_0}$.
From the above argument, we see that
\begin{equation}\label{eq:pi-n}
n_{xy}-|j_0|-1\le \Pi_2(v_{xy})=\Pi_2(w_{xy})\le n_{xy}+|j_0|+1,
\end{equation}
and that either $w_{xy}=v_{xy}$ or $w_{xy}\sim v_{xy}$. The curve 
$\beta$ given by the path
$\cdots\sim v_n\sim v_{n-1}\sim\cdots\sim v_{xy}\sim w_{xy}\sim \cdots\sim w_{n-1}\sim w_n\sim\cdots$ has $\Phi(x)$ 
and $\Phi(y)$ as its end points, and so
\[
d_\rho(\Phi(x),\Phi(y))\le \int_\beta \pi(\beta(t))\, dt
 =\pi(v_{xy})\, \sum_{j=\Pi_2(v_{xy})}^\infty \left[\frac{\pi(v_i)}{\pi(v_{xy})}+\frac{\pi(w_i)}{\pi(w_{xy})}\right].
\]
Note that for $j\ge \Pi_2(v_{xy})$, 
\[
 \eta_-^{j-\Pi_2(v_{xy})}\le \frac{\pi(v_{i})}{\pi(v_{xy})}\le\eta_+^{j-\Pi_2(v_{xy})},\ 
 \text{ and }\  \eta_-^{j-\Pi_2(v_{xy})}\le \frac{\pi(w_i)}{\pi(v_{xy})}\le \eta_+^{j-\Pi_2(v_{xy})}.
\]
Therefore
\[
d_\rho(\Phi(x),\Phi(y))\le \frac{2}{1-\eta_+}\, \pi(v_{xy}).
\]
On the other hand, by~(H3) we have that for all curves $\gamma$ in $X$ that have $\Phi(x)$ and $\Phi(y)$ as their
endpoints (with respect to the metric $d_\rho$),
\[
\int_\gamma\pi(\gamma(t))\, dt\ge K_1^{-1}\, \pi(v_{xy}).
\]
It follows that 
\begin{equation}\label{eq:dRho-vs-pi}
K_1^{-1} \pi(v_{xy})\le d_\rho(\Phi(x),\Phi(y))\le \frac{2}{1-\eta_+}\, \pi(v_{xy}).
\end{equation}
Finally, we note from~\eqref{eq:pi-n} that 
\[
  \eta_-^{n_{xy}}\le  \pi(v_{xy})\le \eta_+^{n_{xy}-1-|j_0|}.
\]
Recall that we choose $n_{xy}$ so that $\alpha^{-n_{xy}}<d(x,y)\le \alpha^{1-n_{xy}}$.
Now the definition of $\tau_+$ and $\tau_-$, together with the choice of $n_{xy}$ above, gives us the validity of~\eqref{eq:biHolder}
with constant $C$ depending only on $\eta_-, \eta_+$, and $j_0$ (which in turn depends only on $\tau$ and $\alpha$).
The last claim of the theorem follows from~\eqref{eq:dRho-vs-pi}.

Note that $Z$ is compact. Therefore,
to prove that $\Phi$ is a homeomorphism, it now suffices to prove surjectivity of $\Phi$. 
Let $(w_k)_k$ be a Cauchy sequence in $(X,d_\rho)$ that is \emph{not} convergent
in $(X,d_\rho)$. By replacing $w_k$ with its nearest vertex if necessary, we may assume without loss of generality that
each $w_k$ is in the vertex set $V$ (for this change in the sequence gives us a Cauchy sequence that is equivalent to the 
original sequence). By passing to a subsequence if necessary, we may also assume that 
for each positive integer $k$,
\begin{itemize}
\item $d_\rho(w_k,w_{k+1})<(K_1^2\alpha)^{-k}$,
\item $\Pi_2(w_{k})<\Pi_2(w_{k+1})$.
\end{itemize}
Indeed, if there is some positive integer $n_0$ such that $\Pi_2(v_k)\le n_0$ for each positive integer $k$, then
the sequence lies in the $d_X$-ball $\{w\in X\, :\, d_X(w,w_0)\le n_0\}$ where $d_X$ is the graph metric on $X$
(obtained by considering path metric with each edge in $X$ to be of unit length). In this case, we would have 
from the proof of Lemma~\ref{lem:metric} that $d$ and $d_\rho$ are biLipschitz on this ball and hence $(w_k)_k$
would be convergent to a point in this $d_X$-ball with respect to $d_X$ and hence with respect to $d_\rho$, violating
our assumption that the sequence is not convergent in $(X,d_\rho)$. Thus the above two conditions can be met by
choosing a subsequence. 

For positive integers $k$ we set $x_k=\Pi_1(w_k)$. Then by the compactness of $Z$
we have that there is some $x_\infty\in Z$ and a subsequence of the sequence $(x_k)_k$, also denoted $(x_k)_k$,
such that $x_k\to x_\infty$ with respect to the metric $d_Z$. For each positive integer $n$ we choose $v_n\in V$
such that $d_Z(\Pi_1(v_n),x_\infty)<\alpha^{-n}$. As in the construction of $\Phi$ we know that $(v_k)_k$ is a
Cauchy sequence with respect to $d_\rho$, and that $\Phi(x_\infty)=[(v_n)_n]_\rho$ (where $[(v_n)_n]_\rho$ denotes the 
collection of all Cauchy sequences in $(X,d_\rho)$ that are equivalent to the Cauchy sequence $(v_n)_n$). We now
show that $(w_k)_k\in[(v_n)_n]_\rho$, for this would conclude the proof of surjectivity of $\Phi$. Since $d_Z(x_k,x_\infty)\to0$
as $k\to\infty$, for each positive integer $n$ we can find $k_n>n$ such that $d_Z(x_\infty,x_{k_n})<\alpha^{-n-1}$.
Then by the choice of $v_k$ we have that $d_Z(\Pi_2(v_n),x_{k_n})<\alpha^{1-n}$. Then with $u_n$ a common ancestor of
$v_n$ and $w_{k_n}$ with the largest value of $\Pi_2(u_n)$, we have from Condition~(H3) and~(H1) that
\[
d_\rho(v_n,w_{k_n})\approx\pi(u_n)\le \eta_+^{\Pi_2(u_n)}\to 0\text{ as }n\to\infty,
\]
the last assertion above following from~\eqref{eq:pi-n}.
It follows that $(w_{k_n})_n$ and $(v_n)_n$ are equivalent Cauchy sequences in $(X,d_\rho)$, completing the proof of
surjectivity of $\Phi$.
\end{proof}

\section{Quasisymmetry}\label{Sec:5}

Recall that a homeomorphism $\Psi:W\to Y$, with $(W,d_W)$ and $(Y,d_Y)$ two metric spaces, is quasisymmetric if there
is a homeomorphism $\eta:(0,\infty)\to(0,\infty)$ with $\lim_{t\to 0^+}\eta(t)=0$ such that for every triple of distinct points
$x_1,x_2,x_3\in W$ we have
\[
\frac{d_Y(\Psi(x_1),\Psi(x_2))}{d_Y(\Psi(x_1),\Psi(x_3))}\le \eta\left(\frac{d_W(x_1,x_2)}{d_W(x_1,x_3)}\right).
\]
Given that $\eta$ is a homeomorphism, it can be seen that $\Psi^{-1}$ is also a quasisymmetry if $\Psi$ is. 
In the event that $W$ (and hence $Y$) is uniformly perfect, then $\eta$ can be chosen to be a power function; there are 
constants $C\ge 1$ and $0<\Theta\le 1$ such that the following choice of $\eta$ works:
\[
\eta(t)=C\, \max\{t^\Theta, t^{1/\Theta}\}.
\]
We refer the interested reader to the discussion on quasisymmetric and quasiconformal maps found in~\cite{Hei}.
We will see in Lemma~\ref{lem:unifPerf} in the next section that when $\rho$ satisfies Conditions~(H1) through~(H4),
$Z$ is necessarily uniformly perfect. We do not assume Condition~(H4) here, and so the space $Z$ need not be
uniformly perfect; however, the quasisymmetric maps we obtain are still of the above-mentioned power function format.

In this section we will focus on quasisymmetric aspects of the map $\Phi$ defined in the previous section. Here
Condition~(H2) plays a vital role. 

\begin{thm}\label{thm:qs}
Suppose that $\rho$ satisfies all three of the conditions~(H1), (H2), and~(H3). Then  the map
\[
\Phi\colon Z\to\partial_\rho X
\]
constructed in Definition~\ref{def:phi} is a quasisymmetric map.
\end{thm}

\begin{proof}
Let $x,y,z$ be three distinct points in $Z$. Then by Theorem~\ref{thm:biHolderBdy}, and in particular, by~\eqref{eq:dRho-vs-pi}, 
we have that
\[
\frac{d_\rho(\Phi(x),\Phi(y))}{d_\rho(\Phi(x),\Phi(z))}\approx \frac{\pi(v_{xy})}{\pi(v_{xz})}.
\]
Suppose first that $\Pi_2(v_{xy})\ge \Pi_2(v_{xz})$.  Let $\gamma$ be a descending path from the root vertex $w_0$
to $x$, passing through $v_{xy}$, and let $\beta$ be a descending path from $w_0$ to $x$, passing through $v_{xz}$.
Then there is a vertex $w$ in the path $\gamma$ such that $\Pi_2(w)=\Pi_2(v_{xz})$; it follows that 
$x\in B_{d_Z}(\Pi_1(v_{xz}),\alpha^{-\Pi_2(v_{x,z})})\cap B_{d_Z}(\Pi_1(w),\alpha^{-\Pi_2(w)})$, and so
$w\sim v_{xz}$. Therefore, by Condition~(H2) and Remark~\ref{rem:pi-is-welpsd}, 
we have that $\pi(v_{xz})\approx\pi(w)$ with comparison constant $K_0$.
Let $\gamma$ be the path $w_0=v_0\sim v_1\sim\cdots\sim v_{xy}$.
It follows that
\[
\frac{d_\rho(\Phi(x),\Phi(y))}{d_\rho(\Phi(x),\Phi(z))}\approx \frac{\pi(v_{xy})}{\pi(w)}=\prod_{j=\Pi_2(v_{xz})}^{\Pi_2(v_{xy})}\rho(w_j)
   \le \eta_+^{\Pi_2(v_{xy})-\Pi_2(v_{xz})}=\alpha^{-\tau_+(\Pi_2(v_{xy})-\Pi_2(v_{xz}))}.
\]
Now by~\eqref{eq:pi-n}, we see that 
\[
\frac{d_\rho(\Phi(x),\Phi(y))}{d_\rho(\Phi(x),\Phi(z))}\lesssim \left(\frac{d_Z(x,y)}{d_Z(x,z)}\right)^{\tau_+}.
\]
Now suppose that $\Pi_2(v_{xy})\le \Pi_2(v_{xz})$. Then, reversing the roles of $y$ and $z$ in the above argument gives
us (with $\beta=(w_0=u_0\sim u_1\sim\cdots)$ and $u$ the vertex in $\beta$ such that $\Pi_2(u)=\Pi_2(v_{xy})$), 
\[
\frac{d_\rho(\Phi(x),\Phi(z))}{d_\rho(\Phi(x),\Phi(y))}\approx \frac{\pi(v_{xz})}{\pi(u)}=\prod_{j=\Pi_2(v_{xy})}^{\Pi_2(v_{xz})}\rho(u_j)
\gtrsim \eta_-^{\Pi_2(v_{xz})-\Pi_2(v_{xy})}
=\alpha^{-\tau_-(\Pi_2(v_{xz})-\Pi_2(v_{xy}))}.
\]
Invoking~\eqref{eq:pi-n} again, we see that
\[
\frac{d_\rho(\Phi(x),\Phi(z))}{d_\rho(\Phi(x),\Phi(y))}\gtrsim \left(\frac{d_Z(x,z)}{d_Z(x,y)}\right)^{\tau_-},
\]
from whence we obtain
\[
\frac{d_\rho(\Phi(x),\Phi(y))}{d_\rho(\Phi(x),\Phi(z))}\lesssim \left(\frac{d_Z(x,y)}{d_Z(x,z)}\right)^{\tau_-}.
\]
Thus $\Phi$ is $\eta$-quasisymmetric with 
\[
\eta(t)\approx \max\{t^{\tau_+}, t^{\tau_-}\}.
\]
\end{proof}

Up to now we have made use of Conditions~(H1), (H2), and~(H3). In the next section we introduce and use Condition~(H4).

\section{Ahlfors regularity}\label{Sec:6}

For each non-negative integer $m$ and $x\in A_m$, and for each positive integer $n$ with $n>m$, we set
$D_n((x,m))$ to be the collection of all vertices $(y,n)\in V$ such that there is a vertically descending path from 
the vertex $(x,m)$ to $(y,n)$. 
Observe that such a path is a sub-path of a vertically descending path from the root vertex $w_0$
to $(y,n)$.

\begin{defn}\label{def:H4}
We say that $\rho:V\to\R$ satisfies Condition~(H4) if there exist $p>0$ and $K_2>0$ such that whenever $x\in A_m$ and
$n>m$, we have
\[
K_2^{-1}\, \pi((x,m))^p\le \sum_{v\in D_n(x,m)}\pi(v)^p\le K_2\, \pi((x,m))^p.
\]
\end{defn}

For the rest of this section we consider the Condition~(H4) in addition to the three conditions given in Definition~\ref{def:H1-3}.

\begin{lem}\label{lem:unifPerf}
Suppose that $\rho$ satisfies Conditions~(H1) through~(H4). Then $(Z,d_Z)$ is uniformly perfect, and
for each $x\in A_n$, $\diam_{d_\rho}\Phi((B_{d_Z}(x,\alpha^{-n})))\approx\pi((x,n))$.
\end{lem}

\begin{proof}
To prove uniform perfectness, it suffices to show that there is some positive integer $N>1$ such that 
for each positive integer $n$ and each $x\in A_n$, the annulus $B_{d_Z}(x,\alpha^{-n})\setminus B_{d_Z}(x,\alpha^{-n-N})$ is 
non-empty. To this end, suppose that $N>2$ is an integer and $x\in A_n$ such that
the annulus $B_{d_Z}(x,\alpha^{-n})\setminus B_{d_Z}(x,\alpha^{-n-N})$ is empty. Then from Condition~(H4) we see that
\[
\pi((x,n))^p\le K_2\pi((x,n+N))^p=K_2\pi((x,n))^p\prod_{j=0}^{N-1}\rho((x,n+j))^p\le K_2\eta_+^{Np}\, \pi((x,n))^p.
\]
It follows that $K_2\eta_+^{Np}\ge 1$. Hence if
\[
N>\frac{1}{p}\, \frac{\log(K_2)}{\log(1/\eta_+)},
\]
then the annulus $B_{d_Z}(x,\alpha^{-n})\setminus B_{d_Z}(x,\alpha^{-n-N})$ must be non-empty. It follows  that
$(Z,d_Z)$ is uniformly perfect, with uniform perfectness constant $C_U=\alpha^N$ where $N$ satisfies the above inequality.

The second claim now follows from the restriction on $N$ given above as well. Indeed, we can find $z\in B_{d_Z}(x,\alpha^{-n})$
such that $d_Z(x,z)\ge \alpha^{-n-N}$. With $v_{xz}$ as in Condition~(H3), we see from~\eqref{eq:pi-n} that
$\alpha^{-\Pi_2(v_{xz})}\approx \alpha^{-n_{xz}}\approx \alpha^{-n}$
and that the graph-distance between the vertices $v_{xy}$ and
$(x,n)$ is bounded by a constant that depends only on the constants $\eta_+,\eta_-, K_0$, and $K_1$.
By~(H2) and~(H1) we have that $\pi((v,n))\approx\pi(v_{xz})$.
Now by the last claim of Theorem~\ref{thm:biHolderBdy} we have that 
\[
\pi(v_{xz})\approx d_\rho(\Phi(x),\Phi(z))\le \diam_{d_\rho}(\Phi(B_{d_Z}(x,\alpha^{-n})).
\]
Now if we choose $w\in B_{d_Z}(x,\alpha^{-n})$ such that 
$\tfrac12\diam_{d_\rho}(\Phi(B_{d_Z}(x,\alpha^{-n}))\le d_\rho(\Phi(x),\Phi(w))$, then
as there is a vertically descending path from the root $w_0$, through $(x,n)$, ending at $\Phi(w)$, it follows that
$v_{xw}$ is a descendant of $(x,n)$; it follows that $\pi(v_{xw})\le \pi((x,n))$, and so by Theorem~\ref{thm:biHolderBdy} again,
\[
\frac12\diam_{d_\rho}(\Phi(B_{d_Z}(x,\alpha^{-n}))\le d_\rho(\Phi(x),\Phi(w))
  \approx\pi(v_{xw})\le \pi((x,n))\approx\pi(v_{xz}).
\]
The combination of the above two inequalities yields the final claim of this lemma.
\end{proof}

From now on we will denote $\Phi(x)$ by $x$ as well whenever $x\in Z$; thus we will
also conflate $\Phi(B_{d_Z}(x,\alpha^{-n})$ with $B_{d_Z}(x,\alpha^{-n})$, as this will not lead to confusion.

\begin{remark}\label{rem:quasiball}
We fix $0<l\le L<\infty$. A set $E\subset Z$ is said to be an $(L,l)$-quasi-ball in $(Z,\theta)$ with center $x\in E$ 
if there is some $\rho>0$
such that $B_\theta(x,l\rho)\subset E\subset B_\theta(x,L\rho)$. Now, for $x\in A_n$, we set
\[
r:=\sup_{y\in B_{d_Z}(x,\alpha^{-n})}\theta(x,y),\qquad \tau:=\inf_{y\in X\setminus B_{d_Z}(x,\alpha^{-n})}\theta(x,y).
\]
By the quasisymmetry of $(Z,\theta)$ with respect to $(Z,d_Z)$, we see that $r\le \eta(1)\, \tau$. If $\tau\ge r$, then
we have that $B_{d_Z}(x,\alpha^{-n})=B_\theta(x,r)$, and so we can take $l=L=1$ and $\rho=r$. If $\tau<r$, then we have that
$\tau<r\le \eta(1)\tau$, and so $B_\theta(x,\tau)\subset B_{d_Z}(x,\alpha^{-n})\subset B_\theta(x,r)\subset B_\theta(x,\eta(1)\tau)$,
and we can then take $\rho=\tau$ and $l=1$, $L=\max\{1,\eta(1)\}$. Thus for each $x\in A_n$ we have that
$B_{d_Z}(x,\alpha^{-n})$ is $(1,\max\{1,\eta(1)\})$-quasi-ball in $(Z,\theta)$ with center $x$.
\end{remark}

\begin{thm}
Suppose that 
$\rho$ satisfies Conditions~(H1), (H2), (H3), and~(H4). Then $(\partial_\rho X, d_\rho)$ is Ahlfors $p$-regular.
\end{thm}

\begin{proof}
To prove the claim we construct an Ahlfors $p$-regular measure on $Z$ as a weak limit of a sequence of measures on $Z$.
We fix a positive integer $n$ and set the measure $\mu_n$ on $Z$ as follows: for each Borel set $E\subset Z$, we set
\[
\mu_n(E):=\sum_{x\in A_n\cap E}\pi((x,n))^p.
\]
Note that, thanks to Condition~(H4), there is a relationship between $\mu_n$ and $\mu_m$ for $n>m$ given by
\[
K_2^{-1}\mu_n(Z)\le \mu_m(Z)\le C\, K_2\mu_n(Z).
\]
Here the constant $C$ is the bounded overlap constant mentioned in Definition~\ref{rem:propties-hyp}~(9).
It follows that for each positive integer $n$, 
\[
0<(CK_2)^{-1}\mu_1(Z)\le \mu_n(Z)\le K_2\mu_1(Z)<\infty, 
\]
and hence
the sequence of measures $(\mu_n)_n$ is tight on $Z$, and so there is a subsequence $(\mu_{n_k})_k$ and a 
Radon measure $\mu$ on $Z$ such that $\mu_{n_k}$ converges weakly to $\mu$; moreover,
$K_2^{-1}\mu_1(Z)\le \mu(Z)\le K_2\mu_1(Z)$. Thus $\mu$ is non-trivial on $Z$.

Note also that for each $x\in Z=\partial_\rho X$, 
\[
d_\rho(w_0,x)\le \sup_\gamma\int_\gamma\pi(\gamma(t))\, dt\le \sum_{n=0}^\infty \eta_+^n=\frac{1}{1-\eta_+}<\infty.
\]
We now wish to show that $\mu$ is Ahlfors $p$-regular on $Z$ with respect to the metric $d_\rho$. 
Since $\Phi$ is a quasisymmetric map from $(Z,d_Z)$ to $(Z,d_\rho)$, it follows that balls in the metric $d_Z$ are quasi-balls
in the metric $d_\rho$, see Remark~\ref{rem:quasiball} above. Hence it suffices to verify the regularity condition for $d_Z$-balls.

Note first that if $n$ is a positive integer and $z\in A_n$, then $\mu_n(B_{d_Z}(z,\alpha^{-n}))=\pi((z,n))^p$. 
We fix $x\in Z$ and $0<r<\tfrac12\diam_{d_Z}(Z)$, and choose the unique positive integer $n_r$ such that 
$\alpha^{-n_r-1}<r\le \alpha^{-n_r}$. Then, for integers $m>n_r+3$, by the definition of $\mu_m$ we have
\[
\mu_m(B_{d_Z}(x,r))=\sum_{z\in A_m\cap B_{d_Z}(x,r)}\, \pi((z,m))^p.
\]
With $z_0\in A_{n_r+2}$ such that $d(x,z_0)<\alpha^{-n_r-2}$, note that necessarily $z_0\in B_{d_Z}(x,r)$. Moreover, for
$m\ge n_r+3$, if $z\in A_m$ such that there is a vertically descending path from $(z_0,n_r+3)$ to $(z,m)$, then
$d(z,z_0)<\alpha^{-n_r-2}$, and so $d(x,z)<\alpha^{-n_r-2}+\alpha^{-n_r-2}<\alpha^{-n_r-1}\le r$, that is,
$z\in A_m\cap B_{d_Z}(x,r)$. Hence $\Pi_1(D_m(z_0,n_r+3))\subset B_{d_Z}(x,r)$, whence we obtain from Condition~(H4) that
\[
\mu_m(B_{d_Z}(x,r))\ge \sum_{v\in D_m(z_0,n_r+3)}\pi(v)^p\ge K_2^{-1}\, \pi(z_0,n_r+3)^p.
\]
Next, let us consider two points $z,w\in A_m\cap B_{d_Z}(x,r)$. Then 
$d(z,w)<\alpha^{1-n_r}$. Let $v_0\sim v_1\sim\cdots\sim v_m=(z,m)$
and $v_0\sim v_1'\sim\cdots\sim v_m'=(w,m)$ be two vertically descending paths from the root vertex $w_0$ to
the vertices $(z,m)$ and $(w,m)$ respectively. Then as $m\ge n_r+3$,  we can find $(z',n_r-1)$ in the first path and
$(w',n_r-1)$ in the second path. We will show that $(z',n_r-1)\sim (w',n_r-1)$. Indeed, 
\[
d(z',w)\le d(z',z)+d(z,w)< \frac{\alpha+1}{\alpha-1}\alpha^{-n_r}+\alpha^{1-n_r}\le\frac{2\alpha}{\alpha-1}\alpha^{1-n_r},
\]
and hence as $\tau\ge 2\alpha/(\alpha-1)$, we 
conclude that $w$ is in both $B_{d_Z}(w',\alpha^{1-n_r})$ and $B_{d_Z}(z', \tau\alpha^{1-n_r})$; that is, $(z',n_r-1)\sim (w',n_r-1)$.
It follows that 
\[
A_m\cap B_{d_Z}(x,r)\subset \bigcup_{(a,n_r-1)\sim (z',n_r-1)}D_m(a,n_r-1).
\]
Hence 
\[
\mu_m(B_{d_Z}(x,r))\le \sum_{(a,n_r-1)\sim (z',n_r-1)}\pi((a,n_r-1))^p.
\]
Given that $Z$ is doubling, the number of vertices $(a,n_r-1)$ that are neighbors of $(z',n_r-1)$ is at most the doubling constant,
and so by Condition~(H1) we have that
\[
\mu_m(B_{d_Z}(x,r))\le C\, \pi((z',n_r-1))^p.
\]
As the graph-distance between $(z',n_r-1)$ and $(z_0,n_r+3)$ is uniformly bounded, (by the doubling property of $Z$ again)
it follows that 
\[
 \mu_m(B_{d_Z}(x,r))\lesssim\, \pi((z_0,n_r+3))^p.
\]

From the above arguments, we obtain that for each positive integer $m\ge n_r+3$, 
\[
\mu_m(B_{d_Z}(x,r))\approx\pi((z_0,n_r+3))^p\approx\pi((z',n_r-1))^p.
\]
Now from Lemma~\ref{lem:unifPerf} together with Condition~(H1),
\[
\diam_{d_\rho}(B_{d_Z}(x,\alpha^{-n_r}))\approx\pi((z_0,n_r+3)),
\]
which completes the proof.
\end{proof}

\section{Obtaining $\rho$ from quasisymmetric change in metric}\label{Sec:7}

The sections prior to this, together, completes the proof of part {\bf I.} of Theorem~\ref{thm:main}. We now 
consider the converse directional claim given in part {\bf II.} of Theorem~\ref{thm:main}.
To do so, we consider a metric $\theta$ on $Z$ such that $(Z,d_Z)$ and $(Z,\theta)$ are quasisymmetric equivalent
with quasisymmetry parametric function $\eta:[0,\infty)\to[0,\infty)$.  In the prior sections we did not
have to assume that $(Z,d_Z)$ is uniformly perfect (with constant $C_U\ge 2$), 
and indeed, the uniform perfectness of $(Z,d_Z)$ followed from
Condition~(H4) needed only in establishing that $d_\rho$ is Ahlfors regular. In this section, however,
we seem to need uniform perfectness of $(Z,d_Z)$ a priori, and then construct a choice of function $\rho$ corresponding
to the quasisymmetry. We will give this construction in Definition~\ref{def:rho} below.

The standing assumptions in this section are that
$(Z,d_Z)$ is compact, doubling, and uniformly perfect, and that $(Z,\theta)$ is quasisymmetric to $(Z,d_Z)$. We will
also use the construction of hyperbolic filling from Section~\ref{Sec:2}, with $\alpha\ge 2$ and
$\tau\ge \max\{\alpha^2+1, C_U^3(C_U^2-4)^{-1}\}$, 
where $C_U$ is the uniform perfectness constant of $(Z,d_Z)$.
We will soon also require $\alpha>C_U^3$, but this requirement is not needed in the first lemma below. 
We reserve the notation
$B(x,r)$ to denote balls in $Z$, centered at $z\in Z$, of radius $r$ with respect to the metric $d_Z$.
Balls with respect to the metric $\theta$ will be denoted $B_\theta(x,r)$.

\begin{lem}\label{lem:QS-nbrs}
Let $x,y\in A_n$ such that $(x,n)\sim (y,n)$. Then 
\[
\diam_\theta(B(x,\alpha^{-n}))\approx\diam_\theta(B(y,\alpha^{-n}))
\]
with the constant of comparison  given by $2\eta(1)\eta(2\tau C_U)$.

Moreover, if $x,z\in Z$ and $n_{xz}$ is the positive integer with $\alpha^{-n_{xz}}<d_Z(x,z)\le \alpha^{1-n_{xz}}$, then
\[
\diam_\theta(B(x,\alpha^{-n_{xz}}))\approx \theta(x,z),
\]
with comparison constant $\max\{2\eta(1),\, \eta(C_U\alpha)\}$.

Finally, if $x\in A_n$, then there exists $y\in B(x,\alpha^{-n})$ such that 
\[
\diam_\theta(B(x,\alpha^{-n}))\approx\theta(x,y)
\]
with comparison constant $2\eta(C_U)$.
\end{lem}

\begin{proof}
The claim in the first part of the lemma follows immediately if $x=y$, so we assume without loss of generality that $x\ne y$. Then
we have that $\alpha^{-n}\le d_Z(x,y)\le 2\tau\alpha^{-n}$.

Let $x_1\in B(x,\alpha^{-n})$ and $\widehat{y_1}\in B(y,\alpha^{-n})$ such that
\begin{align*}
\diam_\theta(B(x,\alpha^{-n}))\le 2\, \theta(x,x_1), 
 \qquad d_Z(\widehat{y_1},y)\ge \alpha^{-n}/C_U.
\end{align*}
Then by the quasisymmetry of the metric $\theta$ with respect to $d_Z$, we see that
\[
\frac{\theta(x_1,x)}{\theta(y,x)}\le \eta\left(\frac{d_Z(x_1,x)}{d_Z(y,x)}\right)\le \eta\left(\frac{\alpha^{-n}}{\alpha^{-n}}\right)=\eta(1).
\]
Therefore, by the choice of $x_1$, we have that
\begin{equation}\label{eq:1a}
\diam_\theta(B(x,\alpha^{-n}))\le 2\eta(1)\, \theta(y,x).
\end{equation}
Again by the quasisymmetry,
\[
\frac{\theta(y,x)}{\theta(\widehat{y_1},y)}\le \eta\left(\frac{2\tau\alpha^{-n}}{\alpha^{-n}/C_U}\right)=\eta(2\tau C_U),
\]
and so
\begin{equation}\label{eq:1b}
\theta(y,x)\le \eta(2\tau C_U)\, \theta(\widehat{y_1},y)\le \eta(2\tau C_U)\, \diam_\theta(B(y,\alpha^{-n})).
\end{equation}
Combining~\eqref{eq:1a} with~\eqref{eq:1b} gives
\[
 \diam_\theta(B(x,\alpha^{-n}))\le 2\eta(1)\eta(2\tau C_U)\, \diam_\theta(B(y,\alpha^{-n})).
\]
The first part of the lemma is now proved by reversing the roles of $x$ and $y$ in the above argument.

To prove the second part of the lemma, note that 
we can find $x_1,\widehat{x_1}\in B(x,\alpha^{-n_{xz}})$
such that
\[
d_Z(x,\widehat{x_1})\ge \alpha^{-n_{xz}}/C_U, \qquad \diam_\theta(B(x,\alpha^{-n_{xz}}))\le 2\theta(x_1,x).
\]
Then
\[
\frac{\theta(x,x_1)}{\theta(x,z)}\le \eta\left(\frac{\alpha^{-n_{xz}}}{\alpha^{-n_{xz}}}\right)=\eta(1),
\]
and so $\diam_\theta(B(x,\alpha^{-n_{xz}}))\le 2\eta(1)\, \theta(x,z)$. On the other hand,
\[
\frac{\theta(x,z)}{\theta(x,\widehat{x_1})}\le\eta\left(\frac{\alpha^{1-n_{xz}}}{\alpha^{-n_{xz}}/C_U}\right)=\eta(C_U\alpha).
\]
It follows that 
\[
\frac{\theta(x,z)}{\eta(C_U\alpha)}\le \theta(x,\widehat{x_1})\le \diam_\theta(B(x,\alpha^{-n_{xz}})).
\]
The second claim of the lemma follows.

To prove the final claim of the lemma, by the use of uniform perfectness we can find a point 
$y\in B(x,\alpha^{-n})\setminus B(x,\alpha^{-n}/C_U)$. Let $\widehat{x}\in B(x,\alpha^{-n})$ such that
$\theta(\widehat{x},x)\ge \diam_\theta(B(x,\alpha^{-n}))/2$. Then
\[
\frac{\theta(\widehat{x},x)}{\theta(y,x)}\le \eta\left(\frac{\alpha^{-n}}{\alpha^{-n}/C_U}\right)=\eta(C_U).
\]
It follows that
\[
\frac12\diam_\theta(B(x,\alpha^{-n}))\le \eta(C_U)\, \theta(y,x)\le \eta(C_U)\, \diam_\theta(B(x,\alpha^{-n})).
\]
\end{proof}

Throughout the rest of this section we also assume that 
\begin{equation}\label{eq:hodgepodge}
C_U>2, \qquad \alpha>C_U^3,\qquad \tau\ge\max\bigg\lbrace\alpha^2+1, \frac{2C_U^3}{C_U^2-4}\bigg\rbrace.
\end{equation}
The first of the above conditions can always be assumed without loss of generality, and the remaining two
conditions merely give us control over the hyperbolic filling parameters $\alpha$ and $\tau$ in terms of $C_U$.
Note that these assumptions are independent of the quasisymmetric metric $\theta$. 
We are now ready to construct the function $\rho:V\to[0,\infty)$ as follows.

As in~\cite{Car} and in the converse statement given in Theorem~\ref{thm:main},
\emph{we assume that $(Z,\theta)$ is also an Ahlfors $p$-regular space for some $p>0$, and set $\mu$ to be
the $p$-dimensional Hausdorff measure on $Z$ induced by the metric $\theta$.}

\begin{defn}\label{def:rho}
We fix a maximal spanning tree $T\subset X$ of the graph $X$ such that $w_0$ is the root of the spanning tree made up
solely of vertical edges, and so that if $[v,w]$ is an edge in $T$ with $w$ a child of $v$, then 
$d_Z(\Pi_1(v),\Pi_1(w))<\alpha^{-\Pi_2(v)}$.
We set $\rho(w_0)=\mu(Z)^{1/p}$, and for each positive integer $n$ and $x\in A_n$ we set
\[
\rho((x,n))=\left(\frac{\mu(B(x,\tau\alpha^{-n}))}{\mu(B(z,\tau\alpha^{1-n}))}\right)^{1/p},
\]
where $z\in A_{n-1}$ such that $(z,n-1)$ is the parent of $(x,n)$ in $T$.
\end{defn}

While the definition of $\rho$ uses the measure $\mu$ associated with the metric $\theta$, the balls $B(x,\tau\alpha^{-n})$
are with respect to the metric $d_Z$. However, note that as $\theta$ is quasisymmetric with respect to $d_Z$, balls
with respect to the metric $d_Z$ are quasi-balls with respect to the metric $\theta$, as seen in Remark~\ref{rem:quasiball}.
The reason for using the balls
with respect to $d_Z$ is that we do not know what the $\theta$-radius of corresponding balls should be.

In~\cite[(2.11)]{KM} and~\cite[Proof of Theorem~3.14(b)]{KM}, (see~\cite[(2.17)]{KM} for yet another variant)
the authors propose an alternative construction of $\rho$ in the absence of Ahlfors regularity:
\[
\rho((x,n)):=\frac{\diam_\theta(B(x,\alpha^{-n}))}{\diam_\theta(B(z,\alpha^{1-n}))},
\]
where $z\in A_{n-1}$ such that $(x,n)$ is a child of $(z,n-1)$ in $T$.  However, their construction
requires the choice of parameter $\alpha$ involved in the hyperbolic filling to be dependent on the quasisymmetry
scaling function $\eta$ in relation to the uniform perfectness constant $C_U$, 
and so gives a slightly different result than that of~\cite{Car}; 
see the comment at the top of page~33 of~\cite{KM}, where the uniform perfectness constant $C_U$ is denoted $K_P$,
$\tau$ is denoted by $\lambda$, and $\alpha$ is denoted by $a$.

\begin{lem}\label{lem:rho-seek-H1}
The function $\rho$ from Definition~\ref{def:rho} satisfies
Conditiions~(H1) of Definition~\ref{def:H1-3}.
\end{lem}

\begin{proof}
Let $(x,n),(z,n-1)\in V$ such that $(x,n)\sim(z,n-1)$ in the tree $T$.
Now, if $w\in B(x,\tau\alpha^{-n})$,
then by triangle inequality and~\eqref{eq:hodgepodge},
\[
d_Z(w,z)<\tau\alpha^{-n}+\alpha^{-n}+\alpha^{1-n}=\left(\frac{\tau+1}{\alpha}+1\right)\alpha^{1-n}<\frac{\tau}{2}\, \alpha^{1-n}.
\]
That is, $B(x,\tau\alpha^{-n})\subset B(z,[1+\tfrac{1+\tau}{\alpha}]\alpha^{1-n})\subset B(z,\tau\alpha^{1-n})$.
By the uniform perfectness of  $(Z,d_Z)$ we can find a point 
$v_0\in B(z,\tfrac{\tau}{2}\alpha^{1-n})\setminus B(z,\tfrac{\tau}{2C_U}\alpha^{1-n})$. We now show that
$B(v_0,\tfrac{\tau}{\alpha^3}\alpha^{1-n})\subset B(z,\tau\alpha^{1-n})\setminus B(x,\tau\alpha^{-n})$.
Indeed, if $y\in B(v_0,\tfrac{\tau}{\alpha^3}\alpha^{1-n})$, then by triangle inequality we have 
\begin{align*}
d_Z(y,z)&<\frac{\tau}{\alpha^3}\alpha^{1-n}+\frac{\tau}{2}\alpha^{1-n}<\tau\alpha^{1-n},\\
d_Z(y,z)&\ge d_Z(z,v_0)-d_Z(v_0,y)\ge \left[\frac{\tau}{2C_U}-\frac{\tau}{\alpha^3}\right]\alpha^{1-n}
  \ge \left[1+\frac{\tau}{\alpha}\right]\alpha^{1-n},
\end{align*}
where we have used~\eqref{eq:hodgepodge} again and the fact that $\alpha^3>\alpha$. Hence
\[
\frac{\mu(B(x,\tau\alpha^{-n}))}{\mu(B(z,\tau\alpha^{1-n}))}
  =1-\frac{\mu(B(z,\tau\alpha^{1-n})\setminus B(x,\tau\alpha^{-n}))}{\mu(B(z,\tau\alpha^{1-n}))}
  \le 1-\frac{\mu(B(v_0,\tfrac{\tau}{\alpha^3}\alpha^{1-n}))}{\mu(B(z,\tau\alpha^{1-n}))}.
\]
As $(Z,d_Z)$ is quasisymmetric to $(Z,\theta)$, we know that balls in the metric $d_Z$ are quasi-balls in the metric $\theta$,
see Remark~\ref{rem:quasiball} above.
Hence, as in the proof of Lemma~\ref{lem:QS-nbrs}, by the Ahlfors regularity of $\mu$ with respect to theta, we have that 
\[
\frac{\mu(B(v_0,\tfrac{\tau}{\alpha^3}\alpha^{1-n}))}{\mu(B(z,\tau\alpha^{1-n}))}\ge c_0
\]
for some constant $0<c_0<1$ that depends on the quasisymmetry function $\eta$ and the Ahlfors regularity constant of $\mu$,
but does not depend on $v_0,x,z,n$. Hence we have that 
\[
\rho((x,n))\le (1-c_0)^{1/p}.
\]
By the Ahlfors regularity of $\mu$, we also have that $\rho(x,n)\ge c_1>0$ with $c_1$ depending only on the 
quasisymmetry parameter $\eta$ and the Ahlfors regularity constant of $\mu$. Thus we can take $\eta_-=c_1$ and
$\eta_+=(1-c_0)^{1/p}$ to complete the proof.
\end{proof}

\begin{lem}\label{lem:H3+H2}
The function $\rho$ satisfies Conditions~(H3) and~(H2).
\end{lem}

\begin{proof}
From the definition of $\rho$, for each $v\in V$ we have that 
\[
\pi(v)=\mu(B(\Pi_1(v),\alpha^{-\Pi_2(v)}))^{1/p}.
\]
Since $d_Z$-balls are quasi-balls in the metric $\theta$ (see Remark~\ref{rem:quasiball}), 
we have by the Ahlfors $p$-regularity of $\mu$ with respect to $\theta$ that
\[
\pi(v)\approx\diam_\theta(B(\Pi_1(v),\alpha^{-\Pi_2(v)})).
\]
Let $u,w\in V$ be two distinct vertices,
and let $v_{uw}\in V$ be as described in Condition~(H3),  
and let $\gamma$ be any curve in $X$
with end points $u,w$. Denoting the vertices in $\gamma$ by $u=u_0\sim u_1\sim\cdots\sim u_k=w$, we have that
\[
\int_\gamma\pi(\gamma(t))\, dt=\sum_{j=0}^{k-1}\pi(u_j)
\approx\sum_{j=0}^{k-1}\diam_\theta\left(B(\Pi_1(u_j),\tau\alpha^{-\Pi_2(u_j)})\right).
\]
For $j=0,\cdots, k-1$ we have that $B(\Pi_1(u_j),\tau\alpha^{-\Pi_2(u_j)})$ intersects 
$B(\Pi_1(u_{j+1}),\tau\alpha^{-\Pi_2(u_{j+1})})$, see the construction given in Definition~\ref{rem:propties-hyp}.
Hence by the triangle inequality and the second part of Lemma~\ref{lem:QS-nbrs}, 
we have that 
\[
\int_\gamma\pi(\gamma(t))\, dt\gtrsim \theta(\Pi_1(u),\Pi_1(w))\approx\diam_\theta(B(\Pi_1(v_{uw}),\alpha^{-\Pi_2(v_{uw})})),
\]
from which the first claim of the lemma follows.

Condition~(H2) now follows from an application of the first part of Lemma~\ref{lem:QS-nbrs}.
\end{proof}

\begin{lem}\label{lem:H4}
The function $\rho$ satisfies Condition~(H4) given in Definition~\ref{def:H4}.
\end{lem}

\begin{proof}
Let $m$ be a positive integer and $x\in A$. We fix a positive integer $n$ such that $n>m$. Note that for each
$v\in D_n(x,m)$ we have that there is a vertically descending path $(x,m)=v_0\sim v_1\sim\cdots\sim v_{m-n}=v$,
and so we have that
\[
d_Z(x,\Pi_1(v))\le \sum_{j=1}^{m-n}d_Z(\Pi_1(v_{j-1}),\Pi_1(v_j))\le \sum_{j=1}^{m-n}\alpha^{-m-j+1}+\alpha^{-m-j}
 \le \frac{2}{\alpha-1}\alpha^{-m},
\]
and so we have that $\Pi_1(D_n(x,m))\subset B(x,A\alpha^{-m})$, where $A=2/(\alpha-1)$.
It follows from the pairwise disjointness property of the balls $B(\Pi_1(v),\alpha^{-n})$
that
\begin{align*}
\sum_{v\in D_n(x,m)}\pi(v)^p = \sum_{v\in D_n(x,m)}\mu(B(\Pi_1(v),\alpha^{-n}))
 &\le \mu(B(x,(A+1)\alpha^{-m}))\\
 &\le C\, \mu(B(x,\alpha^{-m}))\\
 &=C\, \pi((x,m))^p.
\end{align*}
Here we have used the fact that $\mu$ is Ahlfors regular and also from the construction of $\rho$, for each
$z\in A_m$ we have $\mu(B(z,\alpha^{-m}))=\pi((z,m))^p$.

On the other hand, for each $y\in B(x,\alpha^{-m}/2)$ there exists $z\in A_n$ such that $d_Z(z,y)\le \alpha^{-n}$.
It follows from the choice of $\alpha\ge 2$ that 
\[
d_Z(x,z)\le d_Z(x,y)+d_Z(y,z)<\frac{\alpha^{-m}}{2}+\alpha^{-n}\le \frac{\alpha^{-m}}{2}+\frac{\alpha^{-m}}{\alpha}\le \alpha^{-m}.
\]
Therefore $(z,n)\in D_n(x,m)$. It follows that $B(x,\alpha^{-m}/2)\subset\bigcup_{v\in D_n(x,m)}\overline{B}(\Pi_1(v),\alpha^{-n})$,
and so we have 
\begin{align*}
\pi((x,m))^p=\mu(B(x,\alpha^{-m}))\le C \mu(B(x,\alpha^{-m}/2))&\le C\sum_{v\in D_n(x,m)}\mu(B(\Pi_1(v),\alpha^{-n}))\\
  &= C\, \sum_{v\in D_n(x,m)}\pi(v)^p,
\end{align*}
completing the proof.
\end{proof}

\section{An alternative formulation of Conditions (H1)---(H4), and a query}\label{sec:alternate-cond}

A different perspective of the construction in~\cite{Car} 
is to begin with a density function $\omega:V\to[0,1]$ on the vertex set $V$
such that the following four conditions are satisfied:
\begin{enumerate}
\item[{(H1-a)}] There exist $\eta_-, \eta_+$ with $0<\eta_-\le \eta_+<1$ such that for each $v,w\in V$ with $v\sim w$, 
we have
\[
\eta_-\le \frac{\omega(v)}{\omega(w)}\le \eta_+.
\]
\item[{(H2-a)}] There is a constant $K_0>0$ such that whenever $v,w\in V$ with $v\sim w$, we have
\[
\omega(v)\le K_0\, \omega(w).
\]
\item[{(H3-a)}] We extend $\omega$ to edges $v\sim w$ in $X$ linearly by setting $\omega(tv+(1-t)w)=t\omega(v)+(1-t)\omega(w)$,
where $tv+(1-t)w$ is the point on the edge $v\sim w$ that is a distance $t\in[0,1]$ away from $v$.
There is a constant $K_1>0$ such that 
whenever $\gamma$ is a curve in $X$ connecting $x, y\in X$, then 
\[
\int_\gamma \omega(\gamma(t))\, ds\ge K_1\, \omega(v_{xy}).
\]
\item[{(H4-a)}] There exist $p>0$ and $K_2>0$ such that whenever $x\in A_m$ and $n>m$, we have
\[
\frac{1}{K_2}\, \omega((x,m))^p\le \sum_{v\in D_n(x,m)}\omega(v)^p\le K_2\, \omega((x,m))^p.
\]
\end{enumerate}

It is not difficult to see that setting $\rho(v)=\frac{\omega(v)}{\omega(w)}$ where $w$ is any ancestor of $v$ with $w\sim v$,
we have the original four conditions with $\pi\sim\omega$. 
Indeed, Condition~(H1-a) corresponds to Condition~(H1) of~\cite{Car}, Condition~(H2-a) corresponds to 
Condition~(H2) of~\cite{Car}, Condition~(H3-a) corresponds to Condition~(H3) of~\cite{Car}, and
Condition~(H4-a) corresponds to Condition~(H4) of~\cite{Car}.
This perspective allows us to see that $d_\rho$ is actually
a conformal change in the path-metric on the graph $X$.
In this note we chose to use the original formulation of the conditions as found in~\cite{Car}, see Definition~\ref{def:H1-3} and
Definition~\ref{def:H4}, as the purpose of this note is to provide an analysis of~\cite[Theorem~1.1]{Car}. 
However, this perspective helps bridge the gap between the construction proposed in~\cite{Car}
and the conformal changes in metrics associated with a Harnack density $\omega:X\to(0,\infty)$. A density $\omega$ is
a Harnack density if there are constants $C, A\ge 1$ such thatfor $x,y\in X$ with $d(x,y)<A$ we have
\[
\frac{1}{C}\le \frac{\omega(x)}{\omega(y)}\le C.
\]
Given such a density $\omega$, we can equip the (not complete, but locally complete) metric space $(X,d)$ with the new
metric $(X,d_\omega)$ given by
\[
d_\omega(x,y)=\inf_\gamma\, \int_\gamma \omega\, ds,
\]
where $x,y\in X$ and the infimum is over all rectifiable curves in $X$ with end points $x$, $y$. 
The papers~\cite{BBS2, BHK, CKKSS, GS, Her, HD} are some of the many papers in current literature using
such transformations. 
Any density $\omega$ that satisfies the conditions listed at the beginning of this section is automatically
a Harnack density, thanks to~(H2-a).

\noindent {\bf Concluding remarks:}  
The results of~\cite{Car} link the quasisymmetric geometry of $Z$, the boundary of the hyperbolic filling, to the 
metrics on this filling. We note here that in potential theory as well there is a connection between 
nonlocal energy minimization problems on the boundary of compact doublings spaces and local energy minimization
problems in the hyperbolic filling~\cite{CKKSS}; and this connection is given through the perspective of Adams inequality,
on the compactification of the hyperbolic filling, via a measure supported on the boundary $Z$.
In~\cite{BBS} it was shown that if $Z$ is
equipped with a doubling measure, then its hyperbolic filling, modified according to the density 
$\omega_\alpha(x)=\alpha^{-d_X(x,w_0)}$, yields a uniform domain which can be equipped with a lift $\mu_\omega$
of the measure on $Z$ so that
the corresponding metric measure space $X_\alpha:=(X,d_{\omega_\alpha}, \mu_{\omega_\alpha})$ is bounded,
doubling and supports a $1$-Poincar\'e inequality, as does its metric completion (with the zero-extension of the
measure $\mu_{\omega_\alpha}$ to $\partial_{\omega_\alpha}X$). Moreover, the trace of the Sobolev classes on
$X_\alpha$ are certain Besov classes on $Z$. This fact was exploited in~\cite{CKKSS} to study Neumann boundary
value problem on $X_\alpha$ and link it to certain nonlocal fractional operators on $Z$, and one of the key motivating ideas
behind that analysis was an Adams-type inequality~\cite{Ad, AH, Mak}, with the singular measure given by the doubling measure on
$Z=\partial_{\omega_\alpha}X$. Such an inequality was possible because the measure on $Z$ has a co-dimensional
relationship with the measure $\mu_{\omega_\alpha}$ on $X$. If 
$X$ is equipped with the metric $d_\omega$ corresponding to a general density function $\omega$ satisfying
Conditions~(H1-a)---(H3-a) and $Z$ is equipped with a doubling measure, then it would be interesting to know whether
it is possible to lift the measure on $Z$ to $X$ so that a corresponding co-dimensional relationship between the lift and
the measure on $Z$ is valid and supports an Adams inequality, and would indicate a connection
between the study done in~\cite{Car} and nonlinear potential theory as in~\cite{LS}. 
The author recently was able to prove the validity
of Poincar\'e inequality, and as a consequence the Adams inequality (using~\cite{Mak}) under certain additional
conditions on the parameters $\alpha$ and $\tau$.

\end{document}